\documentclass[a4paper,11pt,reqno ]{amsart}

\usepackage[T1]{fontenc}
\usepackage[utf8]{inputenc}
\usepackage{lmodern}
\usepackage[english]{babel}

\usepackage[%
left=2.5cm,       
right=2.5cm,      
top=3.5cm,        
bottom=3.5cm,     
heightrounded,    
bindingoffset=0mm 
]{geometry}

\usepackage{amsrefs}
\usepackage{amsmath}
\usepackage{amssymb}
\usepackage{mathtools}
\usepackage{mathrsfs}
\usepackage[colorlinks=true,linkcolor=blue,urlcolor=blue,citecolor=blue]{hyperref}

\numberwithin{equation}{section}

\usepackage{hyperref} 

\usepackage{tikz}
\usepackage{graphicx}
\usepackage{xcolor}
\usepackage[font=small]{caption}

\usepackage{bookmark}
\usepackage{esint}

\allowdisplaybreaks

\theoremstyle{plain}
\newtheorem{theorem}{Theorem}[section]
\newtheorem{lemma}[theorem]{Lemma}
\newtheorem{proposition}[theorem]{Proposition}

\newtheorem*{mtheorem}{Main Theorem}
\theoremstyle{definition}
\newtheorem{definition}[theorem]{Definition}
\newtheorem{example}[theorem]{Example}
\newtheorem{remark}[theorem]{Remark}
\newtheorem{open.problem}[theorem]{Open Problem}

\newcommand{\Leb}[1]{\mathscr{L}^{#1}} 

\newcommand{\N}{\mathbb{N}}
\newcommand{\R}{\mathbb{R}}

\title[Asymptotics of the $s$-fractional Gaussian perimeter as $s\to 0^+$]
{Asymptotics of the $s$-fractional \\ Gaussian perimeter as $s\to 0^+$}
\author[A.\ Carbotti]{Alessandro Carbotti}
\address{Dipartimento di Matematica
	e Fisica ``E. De Giorgi'', Universit\`a del Salento,
	Via Per Arnesano, 73100 Lecce, Italy.}
\email{alessandro.carbotti@unisalento.it}
\author[S.\ Cito]{Simone Cito}
\address{Dipartimento di Matematica
	e Fisica ``E. De Giorgi'', Universit\`a del Salento,
	Via Per Arnesano, 73100 Lecce, Italy.}
\email{simone.cito@unisalento.it}
\author[D. A. \ La Manna]{Domenico Angelo La Manna}
\address{Department of Mathematics and Statistics, P.O.\ Box 35 (MaD), FI-40014, University of Jyv\"askyl\"a, Finland.}
\email{domenico.a.lamanna@jyu.fi}
\author[D.\ Pallara]{Diego Pallara}
\address{Dipartimento di Matematica
	e Fisica ``E. De Giorgi'', Universit\`a del Salento, and INFN, Sezione di Lecce,
	Via Per Arnesano, 73100 Lecce, Italy.}
\email{diego.pallara@unisalento.it}

\date{\today}  



\begin{document}
	\begin{abstract}
We study the asymptotic behaviour of the renormalised $s$-fractional Gaussian peri\-meter 
of a set $E$ inside a domain $\Omega$ as $s\to 0^+$. Contrary to the Euclidean case, as the Gaussian
measure is finite, the shape of the set at infinity does not matter, but, surprisingly, the limit 
set function is never additive.
\end{abstract}
	
	\maketitle
	
	\tableofcontents

\section{Introduction}\label{sec:intro}

In this paper we consider the fractional Gaussian perimeter
\begin{equation}
\label{eq:fracgaussperimeter}
\begin{split}
P^\gamma_s(E;\Omega):=&\int_{E\cap\Omega}d\gamma(x)\int_{E^c\cap\Omega}K_s(x,y)d\gamma(y) 
\\
&+\int_{E\cap\Omega}d\gamma(x)\int_{E^c\cap\Omega^c}K_s(x,y)d\gamma(y)
+\int_{E\cap\Omega^c}d\gamma(x)\int_{E^c\cap\Omega}K_s(x,y)d\gamma(y), 
\end{split}
\end{equation}
where $\gamma$ is the standard Gaussian measure in $\R^N$ defined in \eqref{eq:gaussianmeasure} and the kernel $K_s$ is the jumping kernel 
defined in \eqref{defK_s} and study the asymptotics of $sP^\gamma_s(E;\Omega)$ as $s\to 0^+$. In this sense 
this is a parallel study of our previous paper \cite{CaCiLaPa2}, where the $\Gamma$-limit of 
$(1-s)P^\gamma_s(E;\Omega)$ as $s\to 1^-$ is studied. 

In the Euclidean setting the notion of $s$-fractional perimeter recovers the classical perimeter 
when $s\to 1^-$ in various senses as proved in \cite{AmDeMa, BouBreMir, CafVal, davila,
lombardini, ponce}. On the other side when $s\to 0^+$ one may wonder if there is convergence to some 
measure related to the Lebesgue one, and actually it holds true when considering the fractional perimeter 
of a set in the whole space (see \cite{MazSha}), but in a domain $\Omega$ the limit of $sP^\gamma_s(E;\Omega)$ does not always exist, 
and when it does, as a function of the set $E$ it is not always a measure as proved in \cite{DiFiPaVa}.

The main result of this paper consists in the computation of the limit
\begin{equation}\label{eq:deflim}
\mu(E):=\lim_{s\to 0^+}sP^\gamma_s(E;\Omega)
\end{equation}
and the analysis of the set function $\mu$. 
The Gaussian case is different from the Euclidean case treated in \cite{DiFiPaVa}. Indeed, the limit in 
\eqref{eq:deflim} always exists under the only assumption that $P^\gamma_{s_0}(E;\Omega)<\infty$ 
for some $s_0\in(0,1)$ and it is not affected neither by the behaviour at infinity of the set $E$ 
nor on the unboundedness and $C^{1,\alpha}$ regularity of $\Omega$. Nevertheless, in the limit 
cases $E\subset\Omega$ or $\Omega=\R^N$ we dot not recover at the limit the Gaussian measure 
of $E$, but rather $2\gamma(E)\gamma(E^c)$, a result that is is coherent with the fact that, 
whenever it exists, $\mu(E)=\mu(E^c)$. A related result in the Euclidean setting is the Maz'ya-Shaposhnikova 
approximation theorem proved in \cite{MazSha} in the framework of fractional Sobolev spaces 
$W^{s,p}(\R^N)$. We prove in Theorem \ref{teo:mazya} an analogous result in the Gaussian case, $p=2$. 
Our result is intrinsecally different with respect to its Euclidean counterpart concerning both 
the methods and the result, since in the Gaussian case the Ornstein-Uhlenbeck operator has 
compact resolvent (hence we can use a series expansion) 
and the constants are eigenfunctions relative to the 0 eigenvalue. 

In the following, we denote by $\mathcal{E}$ the family of sets $E\subset\R^N$ such that the limit in \eqref{eq:deflim} exists which is defined as
\begin{equation*}
\mathcal{E}:=\left\{E\in\mathcal{M}(\R^N)\quad\text{s.t.}\quad\exists s_0\in(0,1)\quad\text{s.t.}\quad P^\gamma_{s_0}(E;\Omega)<\infty \right\}.
\end{equation*}
We stress that, differently from \cite{DiFiPaVa}, we do not need to complement $\mathcal{E}$ with a control of the behaviour at infinity of its elements.
Let us state the main result of the present paper. 

\begin{mtheorem}\label{teo:main}
Let $\Omega\subset\R^N$ an open connected set with Lipschitz boundary. Then for any $E\subset\R^N$ 
measurable set such that $P^\gamma_{s_0}(E;\Omega)<\infty$ for some $s_0\in(0,1)$ the limit 
\eqref{eq:deflim} exists and it holds
\begin{equation}
\label{eq:claimlim}
\mu(E)=2\left(\gamma(E)\gamma(\Omega\setminus E)+\gamma(E\cap\Omega)\gamma(E^c\cap\Omega^c)\right).
\end{equation}
\end{mtheorem}

In Section \ref{sec:notprel} we introduce the main tools and definitions. In Section 
\ref{sec:mainresults} we firstly prove our Main Theorem by stating and proving the ancillary 
Propositions \ref{prop:limsup} and \ref{prop:liminf} and we show some properties 
of the limit set function $\mu$. In the last Section \ref{sec:finalremarks} we prove that for the Gaussian 
fractional perimeter defined and used in \cite{DL} the asymptotics for $s\to 0^+$ is trivial. 

\section{Notation and preliminary results}\label{sec:notprel}
For $N\in\N$ we denote by $\gamma$ the Gaussian measure on $\R^N$
\begin{equation}
\label{eq:gaussianmeasure}
\gamma:=\frac{1}{(2\pi)^{N/2}}e^{-\frac{|\cdot|^2}{2}}\Leb{N},
\end{equation}
where $\Leb{N}$ is the Lebesgue measure. With a little abuse of notation we denote by 
$\gamma$ both the measure and its density with respect to $\Leb{N}$. Moreover, in the sequel we use the measure $\lambda:=\frac{1}{(2\pi)^{N/2}}e^{-\frac{|\cdot|^2}{4}}\Leb{N}$.

In order to define the fractional perimeter, we introduce the Ornstein-Uhlenbeck 
semigroup, its generator $\Delta_\gamma$, the fractional powers of the generator and the functional setting.

\begin{definition}
	Let $t>0$ and $x\in\R^N$. For $u\in L^1_\gamma(\R^N)$ we define the Ornstein-Uhlenbeck semigroup as
	$$
	e^{t\Delta_\gamma}u(x):=\int_{\R^N}M_t(x,y)u(y)d\gamma(y)
	$$
	where $M_t(x,y)$ denotes the Mehler kernel
	$$
	M_t(x,y):=\frac{1}{(1-e^{-2t})^{N/2}}
	\exp\left(-\frac{e^{-2t}|x|^2-2e^{-t}x\cdot y+e^{-2t}|y|^2}{2(1-e^{-2t})}\right),
	$$
	which satisfies
	$$
	e^{t\Delta_\gamma}1=\int_{\R^N}M_t(x,y)d\gamma(y)=1,
	$$
	for any $t>0$ and any $x\in\R^N$.
	
	The generator of $e^{t\Delta_\gamma}$ acts on sufficiently smooth functions as
	\[
	\Delta_\gamma u = \Delta u - x \cdot Du
	\]
	and is called Ornstein-Uhlenbeck operator; see e.g. \cite{LunMetPal} and the references therein 
	for the main properties of $e^{t\Delta_\gamma}$ and $\Delta_\gamma$.

Since $-\Delta_\gamma$ is a positive definite and selfadjoint operator which generates a $C_0$-semigroup 
of contractions in $L^2_\gamma(\R^N)$, we can define its fractional powers by means of spectral 
decomposition via the Bochner subordination formula. 
In particular, for $s\in(0,1)$ and $x\in\R^N$ the fractional Ornstein-Uhlenbeck operator is defined as
	\begin{equation}
	\label{eq:bochnersubordination}
	\begin{split}
	(-\Delta_\gamma)^su(x):&=\frac{1}{\Gamma(-s)}\int_0^{\infty}\frac{e^{t\Delta_\gamma} u(x)-u(x)}
	{t^{s+1}}dt
	\\
	&=\frac{1}{\Gamma(-s)}\int_0^{\infty}\frac{dt}{t^{s+1}}\int_{\R^N}M_t(x,y)(u(y)-u(x))d\gamma(y) \\
	&=\frac{1}{|\Gamma(-s)|}\int_{\R^N}\left(u(x)-u(y)\right)K_{2s}(x,y)d\gamma(y),
	\end{split}
	\end{equation}
	where for $\sigma>0$ we have set 
	\begin{equation}\label{defK_s}
	K_\sigma(x,y):=\int_0^{\infty}\frac{M_t(x,y)}{t^{\frac \sigma2+1}}\, dt,
	\end{equation}
and the right-hand side in \eqref{eq:bochnersubordination} has to be intended in the Cauchy principal 
value sense.
Notice that the integrability of the function 
$$
(0,\infty)\ni t\mapsto \frac{M_t(x,y)}{t^{\frac{sp}{2}+1}}
$$
near zero, for any $x,y\in\R^N$, $x\neq y$, is ensured by the fact that
\begin{equation}
\label{eq:kernelquotient}
\lim_{t\to 0^+}\frac{M_t(x,y)}{H_t(|x-y|)}=(2\pi)^{N/2}e^{\frac{|x|^2}{4}}e^{\frac{|y|^2}{4}}
\quad\text{for any}\quad x,y\in\R^N,
\end{equation}
where, for $r\ge 0$, $H_t$ is the Gauss-Weierstrass kernel $H_t(r):=\frac{e^{-\frac{r^2}{4t}}}{(4\pi t)^{N/2}}$.
\end{definition}


\begin{definition}
	\label{def:gaussfracsobspace}
	Let $s\in(0,1)$ and $1\le p<\infty$. We define the fractional 
	Gaussian Sobolev space $W^{s,p}_\gamma(\R^N)$ as
	$$
	W^{s,p}_\gamma(\R^N):=\left\{u\in L^p_\gamma(\R^N);\ [u]_{W^{s,p}_\gamma(\R^N)}<\infty\right\},
	$$
	where
	$$
	[u]_{W^{s,p}_\gamma(\R^N)}:=\left(\int_{\R^N} d\gamma(x)\int_{\R^N}|u(x)-u(y)|^p
	K_{sp}(x,y)d\gamma(y)\right)^{1/p},
	$$
	and $K_{sp}$ is defined in \eqref{defK_s} with $\sigma=sp$. When $p=2$, as usual we use the notation $H^s_\gamma(\R^N)$ instead of $W^{s,2}_\gamma(\R^N)$.
\end{definition}

For the sake of completeness we recall that the Gaussian perimeter of a measurable set $E$ in a Lipschitz open connected set $\Omega$ is defined by
\begin{equation}
\label{eq:gaussper}
P^\gamma(E;\Omega)=\sup\left\{\int_E\left(\text{div}\,\varphi-\varphi\cdot x\right)\:d\gamma(x):
\varphi\in C^\infty_c(\Omega;\R^N),\ \|\varphi\|_\infty\le 1\right\}.
\end{equation}

Now, we make more precise the definition of Gaussian fractional perimeter 
\eqref{eq:fracgaussperimeter} given in Section \ref{sec:intro}.

\begin{definition}
	Let $\Omega\subset\R^N$ be a connected open set with Lipschitz boundary, and $E\subset\R^N$ a 
	measurable set. We define the Gaussian $s$-perimeter of $E$ in $\Omega$ as
	$$
	P^\gamma_s(E;\Omega):=P^{\gamma,L}_s(E;\Omega)+P^{\gamma,NL}_s(E;\Omega),
	$$
	where the {\em local part} is 
	\begin{equation*}
	P^{\gamma,L}_s(E;\Omega):=\int_{E\cap\Omega}d\gamma(x)\int_{E^c\cap\Omega}K_s(x,y)d\gamma(y),
	\end{equation*}
	and the {\em nonlocal part} is 
	\begin{equation*}
	P^{\gamma,NL}_s(E;\Omega):=\int_{E\cap\Omega}d\gamma(x)\int_{E^c\cap\Omega^c}K_s(x,y)d\gamma(y)+
	\int_{E\cap\Omega^c}d\gamma(x)\int_{E^c\cap\Omega}K_s(x,y)d\gamma(y).
	\end{equation*}
	Using the symmetry of the kernel $K_s$ we immediately notice that $P^\gamma_s(E^c;\Omega)=P^\gamma_s(E;\Omega)$ for any measurable set $E$.
	If $\Omega=\R^N$ we simply write $P^\gamma_s(E)$ instead of $P^\gamma_s(E;\R^N)$. We notice that if $E\subset\Omega$ or $E^c\subset\Omega$ we have that $P^\gamma_s(E;\Omega)=P^\gamma_s(E)$.
\end{definition}

In the sequel, for $A, B$ measurable and disjoint sets, we denote with $L^\gamma_s(A,B)$ the ($s$-Gaussian) interaction functional
\begin{equation}
\label{eq:interaction}
L^\gamma_s(A,B):=\int_Ad\gamma(x)\int_BK_s(x,y)d\gamma(y).
\end{equation}
Using this notation we have 
$$
P^\gamma_s(E;\Omega)=L^\gamma_s(E\cap\Omega,E^c\cap\Omega)+L^\gamma_s(E\cap\Omega,E^c\cap\Omega^c)
+L^\gamma_s(E\cap\Omega^c,E^c\cap\Omega).
$$
It is useful the following integration by parts formula proved for instance in \cite{CaCiLaPa2}
\begin{equation}
\label{eq:intbyparts}
\frac 12[u]^2_{H^s_\gamma(\R^N)}=\int_{\R^N}u(-\Delta_\gamma)^s u\:d\gamma.
\end{equation}
The kernel $K_s$ satisfies the following estimates (see \cite[Lemmas 2.8, 2.9]{CaCiLaPa2}).

\begin{lemma}
	For any $x,y\in\R^N$ and for any $s\in(0,1)$ we have
		\begin{equation}
		\label{eq:fractkernellowerbound}
		K_s(x,y)\ge\frac{C_{N,s}}{|x-y|^{N+s}},
		\end{equation}
		where $C_{N,s}:=2^{s+\frac N2}\Gamma\left(\frac{s+N}{2}\right)$, and 
		\begin{equation}
		\label{eq:upperadialestimate}
		K_s(x,y)\le e^{\frac{|x|^2}{4}}e^{\frac{|y|^2}{4}}\tilde{K}_s(|x-y|),
		\end{equation}
		where, for any $r\ge 0$, $\tilde{K}_s$ denotes the decreasing kernel
		$$
		\tilde{K}_s(r):=\int_0^{\infty} \exp\left(-\frac{e^tr^2}{2(e^{2t}-1)}\right)
		\frac{dt}{t^{\frac s2+1}(1-e^{-2t})^{N/2}}.
		$$
\end{lemma}

\section{Main Results}
\label{sec:mainresults}

We begin this section by proving the analogous of \cite[Theorem 3]{MazSha} in the case $p=2$
in the Gaussian setting. Notice that our proof exploits the Hilbert structure of $H^s_\gamma(\R^N)$
and the compactness of the resolvent of $\Delta_\gamma$. For $p\ne 2$ the proof is more delicate and requires explicit estimates on the kernel joint with a Hardy-type inequality (see \cite[Subsection 2.1]{FraSei}).

\begin{theorem}[Maz'ya-Shaposhnikova approximation in $H^s_\gamma(\R^N)$]\label{teo:mazya}
	Let $s_0\in(0,1)$ and $u\in H^{s_0}_\gamma(\R^N)$. Then it holds that
	\begin{equation*}
	\lim_{s\to 0^+}s[u]^2_{H^s_\gamma(\R^N)}=2\left(\left\|u\right\|^2_{L^2_\gamma(\R^N)}-\left|\int_{\R^N}u\:d\gamma\right|^2\right).
	\end{equation*}
\end{theorem}
\begin{proof}
Let us notice that since $u\in L^2_\gamma(\R^N)$, we can write it in terms of the orthonormal basis of eigenfunctions $\mathcal{B}$ of $(-\Delta_\gamma)^s$ given by Hermite polynomials (see for instance \cite{ErMaObTr}), i.e. $\mathcal{B}=\{H_n\}_{n\in\N_0}$, with $H_0\equiv1$ on $\R^N$. We recall that on the whole of $\R^N$ the spectral fractional Ornstein-Uhlenbeck operator coincides with the integro-differential operator in \eqref{eq:bochnersubordination}, and so, by the spectral mapping Theorem (see e.g. \cite[Theorem 5.3.1]{MarSan}) the latter has discrete spectrum given by $\sigma((-\Delta_\gamma)^s)=\sigma((-\Delta_\gamma))^s=\{n^s\}_{n\in\N_0}.$

With these ideas in mind we have that
$$
u=\sum_{n=0}^\infty(u,H_n)H_n,
$$
$$
(-\Delta_\gamma)^su=|\Gamma(-s)|\sum_{n=1}^\infty n^s(u,H_n)H_n,
$$
By using the integration by parts formula \eqref{eq:intbyparts}
\begin{equation}
\label{eq:mazshap}
s[u]^2_{H^s_\gamma(\R^N)}=2s\int_{\R^N}u(-\Delta_\gamma)^sud\gamma=2\Gamma(1-s)\sum_{n=1}^\infty n^s|(u,H_n)|^2,
\end{equation}
where the right-hand side in \eqref{eq:mazshap} is finite for any $s\in(0,s_0)$ thanks to the assumption $u\in H^{s_0}_\gamma(\R^N)$. Passing to the limit for $s\to 0^+$ in \eqref{eq:mazshap} we have
\begin{align*}
\lim_{s\to 0^+}s[u]^2_{H^s_\gamma(\R^N)}&=2\sum_{n=1}^\infty|(u,H_n)|^2\\
&=2\left[\left(\sum_{n=0}^\infty|(u,H_n)|^2\right)-|(u,H_0)|^2\right]=2\left(\|u\|^2_{L^2_\gamma(\R^N)}-\left|\int_{\R^N}u\:d\gamma\right|^2\right),
\end{align*}
concluding the proof.
	\end{proof}

\begin{remark}
We point out that Theorem \ref{teo:mazya} is sufficient to prove our Main Theorem when $\Omega=\R^N$. Indeed, by choosing $u=\chi_E$, where $E$ is a measurable set with $P^\gamma_{s_0}(E)<\infty$ for some $s_0\in(0,1)$, we get
\begin{align*}
\lim_{s\to 0^+}sP^\gamma_{s}(E)&=\lim_{s\to 0^+}\frac{s}{2}[\chi_E]^2_{H^{\frac s2}_\gamma(\R^N)}=\lim_{\sigma\to 0^+}\sigma[\chi_E]^2_{H^{\sigma}_\gamma(\R^N)}=2\left(\gamma(E)-\gamma(E)^2\right)\\
&=2\gamma(E)(1-\gamma(E))=2\gamma(E)\gamma(E^c).
\end{align*}
\end{remark}

The remaining part of the section is devoted to the proof of our Main Theorem in the general case.

\begin{proposition}\label{prop:limsup}
Let $\Omega\subset\R^N$ an open connected set with Lipschitz boundary and let  $E\subset\R^N$ be measurable. If $P^\gamma_{s_0}(E;\Omega)<\infty$ for some $s_0\in(0,1)$ then 
\begin{equation}
\label{eq:claimlimsup}
\limsup_{s\to 0^+}s P^\gamma_s(E;\Omega)\le 2
\bigl[\gamma(E)\gamma(\Omega\setminus E)+\gamma(E\cap\Omega)\gamma(E^c\cap\Omega^c)\bigr].
\end{equation}
\end{proposition}
\begin{proof}
We split
$$
K_s(x,y)=\int_0^1\frac{M_t(x,y)}{t^{\frac s2+1}}dt+\int_1^{\infty}\frac{M_t(x,y)}{t^{\frac s2+1}}dt.
$$
For the first term we have 
\begin{equation}
\label{eq:pezzopiccolo}
\int_0^1\frac{M_t(x,y)}{t^{\frac s2+1}}dt\le\int_0^1\frac{M_t(x,y)}{t^{\frac{s_0}{2}+1}}dt\le \int_0^{\infty}\frac{M_t(x,y)}{t^{\frac{s_0}{2}+1}}dt=K_{s_0}(x,y),
\end{equation}
for any $x,y\in\R^N$ and $s\le s_0$.
To handle the second term, we write
$$
M_t(x,y)=\frac{\exp\left(\phi_t(x,y)\right)}{(1-e^{-2t})^{N/2}}.
$$
and estimate 
\begin{equation}
\label{eq:stimapezzogrande}
\begin{split}
\exp\left(\phi_t(x,y)\right)\gamma(x)\gamma(y)&=\exp\left(-\frac{e^{-t}|x-y|^2+(|x|^2+|y|^2)(e^{-2t}-e^{-t})}{2(1-e^{-2t})}\right)\gamma(x)\gamma(y)\\
&=\exp\left(-\frac{e^{-t}|x-y|^2}{2(1-e^{-2t})}\right)\exp\left(-\frac{(|x|^2+|y|^2)(e^{-2t}-e^{-t})}{2(1-e^{-2t})}\right)\gamma(x)\gamma(y)\\
&\le\frac{1}{(2\pi)^N}\exp\left(-\frac{|x|^2+|y|^2}{2}\left(\frac{e^{-2t}-e^{-t}}{1-e^{-2t}}+1\right)\right)\\
&=\frac{1}{(2\pi)^N}\exp\left(-\frac{|x|^2+|y|^2}{2}\frac{1}{1+e^{-t}}\right),
\end{split}
\end{equation}
for any $t>0$ and $x,y\in\R^N$.
Now, we split again
$$
\int_1^{\infty}\frac{M_t(x,y)}{t^{\frac s2+1}}dt=\int_1^{1/s}\frac{M_t(x,y)}{t^{\frac s2+1}}dt+\int_{1/s}^{\infty}\frac{M_t(x,y)}{t^{\frac s2+1}}dt.
$$
Using \eqref{eq:stimapezzogrande} we have 
\begin{equation}
\label{eq:stimadaunoaunosuesse}
\begin{split}
s\gamma(x)\gamma(y)\int_1^{1/s}\frac{M_t(x,y)}{t^{\frac s2+1}}dt&\le\frac{s}{(2\pi)^N}\frac{1}{(1-e^{-2})^{N/2}}\exp\left(-\frac{|x|^2+|y|^2}{2}\frac{1}{1+e^{-1}}\right)\int_1^{1/s}\frac{dt}{t^{\frac s2+1}}\\
&=\frac{1}{(2\pi)^N}\frac{2}{(1-e^{-2})^{N/2}}\left(1-s^{s/2}\right)\exp\left(-\frac{|x|^2+|y|^2}{2}\frac{1}{1+e^{-1}}\right)
\end{split}
\end{equation}
and
\begin{equation}
\label{eq:stimadaunosuesseainf}
\begin{split}
s\gamma(x)\gamma(y)\int_{1/s}^{\infty}\frac{M_t(x,y)}{t^{\frac s2+1}}dt&\le\frac{s}{(2\pi)^N}\int_{1/s}^{\infty}\exp\left(-\frac{|x|^2+|y|^2}{2}\frac{1}{1+e^{-t}}\right)\frac{dt}{t^{\frac s2+1}(1-e^{-2t})^{N/2}} \\
&\le\frac{s}{(2\pi)^N}\frac{1}{(1-e^{-\frac 2s})^{N/2}}\exp\left(-\frac{|x|^2+|y|^2}{2}\frac{1}{1+e^{-\frac 1s}}\right)\int_{1/s}^{\infty}\frac{dt}{t^{\frac s2+1}}\\
&=\frac{1}{(2\pi)^N}\frac{2}{(1-e^{-\frac 2s})^{N/2}}\exp\left(-\frac{|x|^2+|y|^2}{2}\frac{1}{1+e^{-\frac{1}{s}}}\right)s^{s/2}.
\end{split}
\end{equation}
By using \eqref{eq:pezzopiccolo}, \eqref{eq:stimadaunoaunosuesse} and \eqref{eq:stimadaunosuesseainf}, 
for any $s\in(0,s_0)$ we obtain
\begin{equation}
\label{eq:limsupfinale}
\begin{split}
sP^\gamma_s(E;\Omega)\le &sP^\gamma_{s_0}(E;\Omega)\\
&+\frac{1}{(2\pi)^N}\frac{2}{(1-e^{-2})^{N/2}}\left(1-s^{s/2}\right)L_{f}(E\cap\Omega,E^c\cap\Omega)\\
&+\frac{1}{(2\pi)^N}\frac{2}{(1-e^{-2})^{N/2}}\left(1-s^{s/2}\right)L_{f}(E\cap\Omega,E^c\cap\Omega^c)\\
&+\frac{1}{(2\pi)^N}\frac{2}{(1-e^{-2})^{N/2}}\left(1-s^{s/2}\right)L_{f}(E\cap\Omega^c,E^c\cap\Omega)\\
&+\frac{s^{s/2}}{(2\pi)^N}\frac{2}{(1-e^{-\frac 2s})^{N/2}}L_{g_s}(E\cap\Omega,E^c\cap\Omega)\\
&+\frac{s^{s/2}}{(2\pi)^N}\frac{2}{(1-e^{-\frac 2s})^{N/2}}L_{g_s}(E\cap\Omega,E^c\cap\Omega^c)\\
&+\frac{s^{s/2}}{(2\pi)^N}\frac{2}{(1-e^{-\frac 2s})^{N/2}}L_{g_s}(E\cap\Omega^c,E^c\cap\Omega),
\end{split}
\end{equation}
where for $A,B$ measurable and disjoint sets and for $0\leq h\in L^1(A\times B)$ we have used the notation 
$$
L_h(A,B)=\int_Adx\int_Bh(x,y)dy,
$$
with $f(x,y):=\exp\left(-\frac{|x|^2+|y|^2}{2}\frac{1}{1+e^{-1}}\right)$ and $g_s(x,y):=\exp\left(-\frac{|x|^2+|y|^2}{2}\frac{1}{1+e^{-\frac{1}{s}}}\right)$.
To conclude, passing to the $\limsup$ as $s\to 0^+$ in \eqref{eq:limsupfinale} it is easily seen that the first four terms in the right hand-side in \eqref{eq:limsupfinale} vanish, and, using the dominated convergence Theorem, the last three ones approach exactly the right-hand side in \eqref{eq:claimlimsup}.
\end{proof}

To complete the asymptotic estimate, we need an estimate from below for the liminf.

\begin{proposition}\label{prop:liminf}
Let $\Omega\subset\R^N$ be an open connected set with Lipschitz boundary. 
Then for any measurable set $E\subset\R^N$ it holds
\begin{equation}
\label{eq:claimliminf}
\liminf_{s\to 0^+}s P^\gamma_s(E;\Omega)\ge 2
\bigl[\gamma(E)\gamma(\Omega\setminus E)+\gamma(E\cap\Omega)\gamma(E^c\cap\Omega^c)\bigr].
\end{equation}
\end{proposition}
\begin{proof}
Let $\delta>0$ and let $R>0$ be such that
\begin{equation}\label{eq:sceltaR}
\begin{split}
\gamma\left((E\cap\Omega)\cap B_R(0)\right)\ge\gamma\left(E\cap\Omega\right)-\delta,&\quad\gamma\left((E^c\cap\Omega)\cap B_R(0)\right)\ge\gamma\left(E^c\cap\Omega\right)-\delta,\\
\gamma\left((E\cap\Omega^c)\cap B_R(0)\right)\ge\gamma\left(E\cap\Omega^c\right)-\delta,&\quad\gamma\left((E^c\cap\Omega^c)\cap B_R(0)\right)\ge\gamma\left(E^c\cap\Omega^c\right)-\delta.
\end{split}
\end{equation}
For any $x,y\in B_R(0)$ it holds
\begin{equation}\label{eq:lowestexp}
\begin{split}
\exp(\phi_t(x,y))\ge\exp\left(-\frac{e^{-2t}|x-y|^2}{2(1-e^{-2t})}\right)\ge\exp\left(-\frac{2e^{-2t}}{1-e^{-2t}}R^2\right),
\end{split}
\end{equation}
where $\phi_t$ is as in \eqref{eq:stimapezzogrande} and we used that $|x-y|^2\le4R^2$. Since
$$\frac{1}{(1-e^{-2t})^{N/2}}>1$$
and the map
$$t\mapsto \exp\left(-\frac{2e^{-2t}}{1-e^{-2t}}R^2\right)$$
is increasing in $(0,+\infty)$ and
by \eqref{eq:lowestexp} we get, for any $x,y\in B_R(0)$,
\begin{equation}\label{eq:lowestker}
\begin{split}
K_s(x,y)&\ge\int_{1/s}^{\infty}\frac{M_t(x,y)}{t^{\frac s2+1}}dt
\\
&\ge\int_{1/s}^{\infty}\frac{1}{t^{\frac s2+1}(1-e^{-2t})^{N/2}}
\exp\left(-\frac{2e^{-2t}}{1-e^{-2t}}R^2\right)dt
\\
&\ge\exp\left(-\frac{2e^{-2/s}}{1-e^{-2/s}}R^2\right)\int_{1/s}^{\infty}\frac{dt}{t^{\frac s2+1}}
=\frac{2}{s}\exp\left(-\frac{2e^{-2/s}}{1-e^{-2/s}}R^2\right) s^{s/2}.
\end{split}
\end{equation}
We can now estimate from below $sP^\gamma_s(E;\Omega)$
\begingroup
\allowdisplaybreaks
\begin{align*}
sP^\gamma_s(E;\Omega)\geq &\, 
s\int_{(E\cap\Omega)\cap B_R(0)}d\gamma(x)\int_{(E^c\cap\Omega)\cap B_R(0)}K_s(x,y)d\gamma(y) 
\\
&+s\int_{(E\cap\Omega)\cap B_R(0)}d\gamma(x)\int_{(E^c\cap\Omega^c)\cap B_R(0)}K_s(x,y)d\gamma(y)
\\
&+s\int_{(E\cap\Omega^c)\cap B_R(0)}d\gamma(x)\int_{(E^c\cap\Omega)\cap B_R(0)}K_s(x,y)d\gamma(y)
\\
\ge\, & 2\exp\left(-\frac{2e^{-2/s}}{1-e^{-2/s}}R^2\right) s^{s/2}
\bigg[\gamma\left((E\cap\Omega)\cap B_R(0)\right)\gamma\left((E^c\cap\Omega)\cap B_R(0)\right)
\\
&+\gamma\left((E\cap\Omega)\cap B_R(0)\right)\gamma\left((E^c\cap\Omega^c)\cap B_R(0)\right)
\\
&+\gamma\left((E\cap\Omega^c)\cap B_R(0)\right)\gamma\left((E^c\cap\Omega)\cap B_R(0)\right)\bigg]
\\
\geq\, &2\exp\left(-\frac{2e^{-2/s}}{1-e^{-2/s}}R^2\right) s^{s/2}
\bigg[\left(\gamma\left(E\cap\Omega\right)-\delta\right)\left(\gamma\left(E^c\cap\Omega\right)-\delta\right)
\\
&+\left(\gamma\left(E\cap\Omega\right)-\delta\right)\left(\gamma\left(E^c\cap\Omega^c\right)-\delta\right)
\\
&+\left(\gamma\left(E\cap\Omega^c\right)-\delta\right)
\left(\gamma\left(E^c\cap\Omega\right)-\delta\right)\bigg]
\\
=\, & 2\exp\left(-\frac{2e^{-2/s}}{1-e^{-2/s}}R^2\right) s^{s/2}
\\
&\cdot\left[\gamma(E)\gamma(\Omega\setminus E)+\gamma(E\cap\Omega)\gamma(E^c\cap\Omega^c)+3\delta^2-(1+\gamma(\Omega))\delta\right]
\\
\ge &\, 2 \exp\left(-\frac{2e^{-2/s}}{1-e^{-2/s}}R^2\right) s^{s/2}
\\
&\cdot\left[\gamma(E)\gamma(\Omega\setminus E)+
\gamma(E\cap\Omega)\gamma(E^c\cap\Omega^c)+3\delta^2-2\delta\right].
\end{align*}
\endgroup
By letting $s\to 0^+$ we obtain
$$\liminf_{s\to 0^+}sP^\gamma_s(E;\Omega)\ge 2\left[\gamma(E)\gamma(\Omega\setminus E)+\gamma(E\cap\Omega)\gamma(E^c\cap\Omega^c)+3\delta^2-2\delta\right],$$
thus we get \eqref{eq:claimliminf} in view of the arbitrariness of $\delta>0$.
\end{proof}

\begin{proof}[Proof of the Main Theorem]
It is an immediate consequence of Proposition \ref{prop:limsup} and Proposition \ref{prop:liminf}.
\end{proof}

In the proof of our Main Theorem, the hypothesis $P^\gamma_{s_0}(E;\Omega)<+\infty$ for some $s_0\in(0,1)$ 
is crucial (it is required to prove Proposition \ref{prop:limsup}). Adapting \cite[Example 2.10]{DiFiPaVa}, we show that there are measurable sets that do not satisfy that requirement.

\begin{example}[A measurable set with $P^\gamma_{s}(E;\Omega)=+\infty$ for any $s\in(0,1)$]
Let us consider a decreasing sequence $(\beta_k)_k\subset\R$ with $\beta_k>0$ for any $k\in\N$ such that
$$M:=\sum_{k=1}^{+\infty}\beta_k<+\infty$$
but
$$\sum_{k=1}^{+\infty}\beta_k^{1-s}=+\infty$$
(in \cite[Example 2.10]{DiFiPaVa} the authors suggest the possible choice $\beta_1=\frac{1}{\log^2 2}$ and $\beta_k=\frac{1}{k\log^2 k}$ for any $k\ge 2$).
Let us define
$$
\Omega:=(0,M)\subset\R,\quad \sigma_m:=\sum_{k=1}^{m}\beta_k,\quad I_m:=(\sigma_m,\sigma_{m+1}),\quad E:=\bigcup_{j=1}^{+\infty}I_{2j}.
$$
We claim that $P^\gamma_{s}(E;\Omega)=+\infty$ for any $s\in(0,1)$. By recalling that $E\subset\Omega$ it holds
\begin{align*}
P^\gamma_{s}(E;\Omega)=P^\gamma_{s}(E)&\ge C_{1,s}\sum_{j=1}^{+\infty}\int_{\sigma_{2j}}^{\sigma_{2j+1}}d\gamma(x)\int_{\sigma_{2j+1}}^{\sigma_{2j+2}}\frac{d\gamma(y)}{|x-y|^{1+s}}\\
&\ge \frac{1}{2\pi}\frac{e^{-M^2}}{s(1-s)}\sum_{j=1}^{+\infty}\left[(\sigma_{2j+2}-\sigma_{2j+1})^{1-s}+(\sigma_{2j+1}-\sigma_{2j})^{1-s}-(\sigma_{2j+2}-\sigma_{2j})^{1-s}\right]\\
&=\frac{1}{2\pi}\frac{e^{-M^2}}{s(1-s)}\sum_{j=1}^{+\infty}\left[\beta_{2j+2}^{1-s}+\beta_{2j+1}^{1-s}-(\beta_{2j+2}+\beta_{2j+1})^{1-s}\right]
\end{align*}
where in the first inequality we used \eqref{eq:fractkernellowerbound}, while in the second inequality we used that $C_{1,s}\ge 1$, the boundedness from below of the Gaussian weights in $(\sigma_{2j},\sigma_{2j+1})\times(\sigma_{2j+1},\sigma_{2j+2})$ for any $j\ge 1$ and that for $a<b\le c<d$
$$
\int_{a}^{b}dx\int_{c}^{d}\frac{dy}{|x-y|^{1+s}}=\frac{1}{s(1-s)}
\left[(c-a)^{1-s}+(d-b)^{1-s}-(c-b)^{1-s}-(d-a)^{1-s}\right].
$$
Since the map $t\mapsto(1+t)^{1-s}$ is concave in $[0,1)$, it holds
$$
1+t^{1-s}-(1+t)^{1-s}\ge st^{1-s}.
$$
By the choice $t=\frac{\beta_{2j+2}}{\beta_{2j+1}}$ we get
$$
\beta_{2j+2}^{1-s}+\beta_{2j+1}^{1-s}-(\beta_{2j+2}+\beta_{2j+1})^{1-s}\ge s\beta_{2j+2}^{1-s}
$$
and so
$$
P^\gamma_{s}(E;\Omega)\ge\frac{1}{2\pi}\frac{e^{-M^2}}{1-s}\sum_{j=1}^{+\infty}\beta_{2j+2}^{1-s}=+\infty,
$$
concluding the proof of the claim.
\end{example}
Now we state some properties of the set function $\mu$.
\begin{proposition}\label{prop:mu}
	$\mu$ is subadditive on $\mathcal{E}$, i.e. $\mu(E\cup F)\le\mu(E)+\mu(F)$ for any $E, F\in\mathcal{E}$; $\mu$ is not monotone with respect to inclusions.
\end{proposition}
\begin{proof}
	To show the subadditivity, we proceed as in the proof of \cite[Proposition 2.1]{DiFiPaVa}; to show the lack of monotonicity, it is sufficient to choose as $E$ a small ball contained in $\Omega$ or a halfspace such that $\mathcal{H}^{N-1}(\partial E\cap\Omega)>0$ and $F=\R^N$.
\end{proof}

Notice that $\mu$ is not additive. Indeed, if $\Omega=\R^N$, then, for any pair of measurable disjoint sets $A,B\subset\R^N$
\begin{align*}
\mu(A\cup B)&=2\gamma(A\cup B)\gamma(A^c\cap B^c)=2\left(\gamma(A)+\gamma(B)\right)\left(1-\gamma(A)-\gamma(B)\right)\\
&=2\gamma(A)\left(1-\gamma(A)\right)+2\gamma(B)\left(1-\gamma(B)\right)-4\gamma(A)\gamma(B)\\
&=2\gamma(A)\gamma(A^c)+2\gamma(B)\gamma(B^c)-4\gamma(A)\gamma(B)=\mu(A)+\mu(B)-4\gamma(A)\gamma(B).
\end{align*}
Otherwise, if $\Omega\neq\R^N$, we proceed as in the proof \cite[Proposition 2.3]{DiFiPaVa} by using the following result.

\begin{lemma}
	\label{lem:posinter}
	For any $A$, $B\subset\R^N$ measurable disjoint sets there exists $C=C(A,B)>0$ such that
	$$
	sL^\gamma_s(A,B)\ge C,
	$$
	for any $s\in(0,1)$.
\end{lemma}
\begin{proof}
We firstly assume that $A$, $B$ are bounded and fix $R>0$ sufficiently large such that $A,B\subset B_R$. 
We have 
\begin{equation}
\begin{split}
sL^\gamma_s(A,B)&\ge s\int_Ad\gamma(x)\int_Bd\gamma(y)\int_1^\infty\frac{M_t(x,y)}{t^{\frac s2+1}}dt \\
&\ge s\int_Ad\gamma(x)\int_Bd\gamma(y)\int_1^\infty\exp\left(-\frac{e^{-2t}(|x|^2+|y|^2)-2e^{-t}(x,y)}{2(1-e^{-2t})}\right)\frac{dt}{t^{\frac s2+1}} \\
&\ge s\int_Ad\gamma(x)\int_Bd\gamma(y)\int_1^\infty\exp\left(-\frac{e^{-2t}|x-y|^2}{2(1-e^{-2t})}\right)\frac{dt}{t^{\frac s2+1}} \\
&\ge s\exp\left(-\frac{2R^2}{(e^2-1)}\right)\int_Ad\gamma(x)\int_Bd\gamma(y)\int_1^\infty\frac{dt}{t^{\frac s2+1}} \\
&=2\exp\left(-\frac{2R^2}{(e^2-1)}\right)\gamma(A)\gamma(B)=:C(A,B).
\end{split}
\end{equation}
If $A$, $B$ are unbounded we simply have 
$$
sL^\gamma_s(A,B)\ge sL^\gamma_s(A\cap B_R,B\cap B_R)\ge C
$$
for any $s\in(0,1)$ and $R>0$.
	\end{proof}

\begin{remark}
	We notice that, even if we add in Lemma \ref{lem:posinter} the hypothesis of strictly positive distance between $A$ and $B$, the result is left unchanged.
\end{remark}

\section{Final remarks}
\label{sec:finalremarks}
We conclude by studying the asymptotics for $s\to 0^+$ even for the fractional perimeter defined in \cite{DL}
\begin{align}
\label{eq:derosalmanna}
\mathcal{J}^\lambda_s(E;\Omega):=&\int_{E\cap\Omega}d\lambda(x)
\int_{E^c\cap\Omega}\frac{d\lambda(y)}{|x-y|^{N+s}} 
\\ \nonumber
&+\int_{E\cap\Omega}d\lambda(x)\int_{E^c\cap\Omega^c}\frac{d\lambda(y)}{|x-y|^{N+s}}
+\int_{E\cap\Omega^c}d\lambda(x)\int_{E^c\cap\Omega}\frac{d\lambda(y)}{|x-y|^{N+s}},
\end{align}
We recall that the functional \eqref{eq:derosalmanna} is linked to \eqref{eq:fracgaussperimeter} by the fact that they have the same $\Gamma$-limit by multiplying by $1-s$ and letting $s\to 1^-$ (\cite[Main Theorem]{CaCiLaPa2}); this depends on the fact that $K_s(x,y)\gamma(x)\gamma(y)$ and $\frac{\lambda(x)\lambda(y)}{|x-y|^{N+s}}$ approach the Dirac delta in the same way, up to a multiplicative constant, when $|x-y|\to 0$. Nevertheless, definition \eqref{eq:derosalmanna} is somehow unnatural, because it is not linked to functional calculus as \eqref{eq:fracgaussperimeter}. Therefore, we can say that \eqref{eq:fracgaussperimeter} is the fractional counterpart of the Gaussian perimeter \eqref{eq:gaussper}, and we can refer to it as  ``Fractional Gaussian perimeter'', while \eqref{eq:derosalmanna} is a weighted version of the fractional perimeter defined in \cite{CafRoqSav}, and we can refer to it as  ``Gaussian fractional perimeter''. As already said in Section \ref{sec:intro} for the Gaussian fractional perimeter the asymptotics for $s\to 0^+$ is not meaningful. Indeed the following Proposition holds.
\begin{proposition}
For any measurable set $E$ such that $\mathcal{J}^\lambda_{s_0}(E;\Omega)<\infty$ for some $s_0\in(0,1)$ we have 
$$
\lim_{s\to 0^+}s\mathcal{J}^\lambda_s(E;\Omega)=0.
$$
\end{proposition}
\begin{proof}
	Let $A$, $B$ be measurable and disjoint sets such that $L^\lambda_{s_0}(A,B)<\infty$ for some 
	$s_0\in(0,1)$, where
	$$
	L^\lambda_{\sigma}(A,B):=\int_Ad\lambda(x)\int_B\frac{d\lambda(y)}{|x-y|^{N+\sigma}}.
	$$ 
	Then, for any $s\in(0,s_0)$ we have 
	\begin{equation}
	\begin{split}
	L^\lambda_s(A,B)&=\iint_{(A\times B)\cap\{|x-y|\ge 1\}}\frac{d\lambda(y)}{|x-y|^{N+s}}d\lambda(x)+\iint_{(A\times B)\cap\{|x-y|< 1\}}\frac{d\lambda(y)}{|x-y|^{N+s}}d\lambda(x) \\
	&\le\lambda(A)\lambda(B)+\iint_{(A\times B)\cap\{|x-y|< 1\}}\frac{d\lambda(y)}{|x-y|^{N+s_0}}d\lambda(x)  \\
	&\le\lambda(A)\lambda(B)+L^\lambda_{s_0}(A,B)<\infty.	
	\end{split}
	\end{equation}
	Therefore
	\begin{equation}
	\label{eq:nulllimit}
	\lim_{s\to 0^+}sL^\lambda_s(A,B)=0.
	\end{equation}
	By applying \eqref{eq:nulllimit} to the couples of sets $(E\cap\Omega,E^c\cap\Omega)$, $(E\cap\Omega,E^c\cap\Omega^c)$, $(E\cap\Omega^c,E^c\cap\Omega)$, we completely prove the claim.
\end{proof}

\begin{remark}
	We notice that even in this case we cannot drop the condition $\mathcal{J}^\lambda_{s_0}(E;\Omega)<\infty$ for some $s_0\in(0,1)$. Indeed \cite[Example 2.10]{DiFiPaVa} still works with
	$$\mathcal{J}^\lambda_{s}(E;\Omega)\ge\frac{1}{2\pi}\frac{e^{-\frac{M^2}{2}}}{1-s}\sum_{j=1}^{+\infty}\beta_{2j+2}^{1-s}=+\infty,$$
	for any $s\in(0,1)$.
\end{remark}
\paragraph*{\bf Funding}\ \\
A.C. has been partially supported by the TALISMAN project Cod. ARS01-01116. S.C. has been partially supported by the ACROSS project CUPF36C18000210005. 
D.A.L. has been supported by the Academy of Finland grant 314227. 
D.P. is member of G.N.A.M.P.A. of the Italian Istituto Nazionale di Alta Matematica (INdAM) and has been partially supported by the PRIN 2015 MIUR project 2015233N54.

\end{document}